\documentclass[a4paper]{article}
\usepackage{amssymb}
\usepackage{amsmath}
\usepackage{amsthm}
\usepackage{hyperref}
\usepackage{enumerate}
\usepackage{url}
\usepackage[T1]{fontenc}

\newtheorem{Theorem} {Theorem} [section]
\newtheorem{Proposition} [Theorem] {Proposition}
\newtheorem{Lemma} [Theorem] {Lemma}
\newtheorem{Corollary} [Theorem] {Corollary}

\newcommand{\cC}{{\mathcal C}}

\newcommand{\cG}{{\mathcal G}}

\newcommand{\cK}{{\mathcal K}}
\newcommand{\cL}{{\mathcal L}}
\newcommand{\cP}{{\mathcal P}}

\newcommand{\cT}{{\mathcal T}}
\newcommand{\II}{{\text{I}}}

\newcommand{\AG}{\text{AG}}
\newcommand{\PG}{\text{PG}}

\newcommand{\<}{\langle}
\renewcommand{\>}{\rangle} 
\renewcommand{\phi}{\varphi} 

\title{Non-Geometric Cospectral Mates of Line Graphs with a Linear Representation}
\author{
Ferdinand Ihringer
}
\date{23 Apr 2022}

\begin{document}
\maketitle

\begin{abstract}
  For an incidence geometry $\cG = (\cP, \cL, \II)$ 
  with a linear representation $\cT_n^*(\cK)$,
  we apply WQH switching to construct a non-geometric 
  graph $\Gamma'$ cospectral with the line graph $\Gamma$ of $\cG$.
  
  As an application, we show that for $h \geq 2$ and $0 < m < h$,
  there are strongly regular graphs with parameters
  $(v, k, \lambda, \mu) = (2^{2h}  (2^{m+h}+2^m-2^h),
    2^h  (2^h+1)(2^m-1),
    2^h  (2^{m+1}-3), 
    2^h  (2^m-1))$
    which are not point graphs of partial geometries
    of order $(s,t,\alpha) = ((2^h+1)(2^m-1), 2^h-1, 2^m-1)$.
\end{abstract}

\section{Introduction}

Let $\cG = (\cP, \cL, \II)$ be a partial linear space of order $(s, t)$,
that is two points are incident with at most one line, each point is
incident with $t+1$ lines, and each line is incident with $s+1$ points.
If for any anti-flag $(P, L)$, there
are precisely $\alpha$ lines through $P$ meeting $L$, then $\cG$ is called a
{\it partial geometry} with parameters $(s, t, \alpha)$.
The line graph $\Gamma(\cL)$ of a partial linear space $\cG$ has vertex set $\cL$,
two lines adjacent when they meet.

Linear representations are an important source for partial linear spaces:
Let $n \geq 2$ and $q$ be a prime power.
Let $\cP$ be the points of $\AG(n+1, q)$ and $\cK$ be a set of points 
in the hyperplane $H \cong \PG(n, q)$ at infinity. 
Let $\cL$ be the lines of $\AG(n+1, q)$ which meet $H$ in a point of $\cK$.
Then $\cT_n^*(\cK)$ denotes the incidence geometry $\cG = (\cP, \cL, \II)$
where incidence is inherited from $\AG(n+1, q)$.
We call $\cT_n^*(\cK)$ the {\it linear representation} of $\cG$.
The line graph $\Gamma(\cL)$ has $|\cK| \cdot q^{n}$ vertices and 
degree $q \cdot (|\cK|-1)$.
We refer to \cite{DeClerck2003,DeWinter2005,Thas1974}
for various constructions of interesting geometries using linear 
representation.

If $q=2^h$, $0 < m < h$, and $n=2$, then a maximal arc $\cK$
of Denniston type of size $(2^h+1)(2^m-1)+1$ (see \cite{Denniston1969}) yields
a partial geometry with linear representation $\cT_2^*(\cK)$
and parameters $(s,t,\alpha) = (2^h-1, (2^h+1)(2^m-1), 2^m-1)$.
In particular, for $m=1$ we obtain generalized quadrangles
of order $(s,t) = (q-1, q+1)$.

Recall WQH-switching \cite{WQH2019} (also see \cite{IM2019}):

\begin{Lemma}[WQH-Switching]\label{lem:WQH}
    Let $\Gamma$ be a graph with vertex set $X$ and
    let $\{ C_1, C_2, D \}$ be a partition of $X$,
    where the subgraphs induced on $C_1$, $C_2$,
    and $C_1 \cup C_2$ are regular, and $C_1$ and $C_2$
    have the same size and degree. Suppose that
    $x \in D$ either has the same number of neighbors
    in $C_1$ and $C_2$, or satisfies 
    $\Gamma(x) \cap (C_1 \cup C_2) \in \{ C_1, C_2 \}$.
    Construct a new graph $\Gamma'$ by interchanging 
    adjacency and nonadjacency between $x \in D$ and $C_1 \cup C_2$
    when $\Gamma(x) \cap (C_1 \cup C_2) \in \{ C_1, C_2\}$.
    Then $\Gamma$ and $\Gamma'$ are cospectral.
\end{Lemma}

Only for this document, let us call a partial linear space $\cG = (\cP, \cL, \II)$ 
of order $(s, t)$ an {\it incomplete $(s, t, \alpha)$-geometry} 
when for any anti-flag $(P, L)$ of $\cG$,
$P$ is collinear with at most $\alpha$ points on $L$.
We show the following:

\begin{Proposition}\label{prop:switching}
    Let $\cG = (\cP, \cL, \II)$ be an incomplete $(q-1, t, \alpha)$-geometry with 
    linear representation $\cT^*_n(\cK)$ 
    and line graph $\Gamma = \Gamma(\cL)$ (so $t+1 = |\cK|$).
    Suppose that $t > q(\alpha-1)$, and that 
    $\cK$ contains a line $K$ with $|K \cap \cK| \geq 2$
    and points $Q_1, Q_2 \in \cK \setminus K$ with $\< Q_1, Q_2 \> \cap K \notin \cK$.
    Then there exists a graph $\Gamma'$ cospectral with $\Gamma$ 
    such that $\Gamma'$ is not the line graph of 
    an incomplete $(s', t, \alpha')$-geometry for any $s', \alpha'$.
\end{Proposition}

A graph (not edgeless, not complete) of order $v$ and degree $k$ is called strongly regular
with parameters $(v, k, \lambda, \mu)$ if any two adjacent vertices
have precisely $\lambda$ common neighbors,
and any two nonadjacent vertices have precisely $\mu$ common neighbors.
A partial geometry of order $(s, t, \alpha)$ yields a strongly regular
graph with parameters
\[
 (v, k, \lambda, \mu) = ( \tfrac{(s+1)(st+\alpha)}{\alpha}, s(t+1), s-1+t(\alpha-1), \alpha(t+1)).
\]
By applying Proposition \ref{prop:switching} to the 
partial geometries from arcs $\cK$ of Denniston type with parameters 
$(s,t,\alpha) = (2^h-1, (2^h+1)(2^m-1), 2^m-1)$ mentioned 
above, we obtain the following:

\begin{Corollary}\label{cor:srg_appl}
    For $h \geq 2$, and $0 < m < h$, 
    there exists a strongly regular graph with parameters 
    $(v, k, \lambda, \mu) = (2^{2h}  (2^{m+h}+2^m-2^h),
    2^h  (2^h+1)(2^m-1),
    2^h  (2^{m+1}-3), 
    2^h  (2^m-1))$
    which is not the line graph of a partial geometry
    of order $(s,t,\alpha) = (2^h-1, (2^h+1)(2^m-1), 2^m-1)$.
\end{Corollary}
\begin{proof}
    From $q=2^h$, $t=(2^h+1)(2^m-1)$, 
    and $\alpha=2^m-1$, the inequality $t > q(\alpha-1)$ follows.
    A line of $\PG(2, q)$ intersects a complete arc of size 
    $(2^h+1)(2^m-1)$ either in $0$ or $2^m$ points, so 
    $K$, $Q_1$, $Q_2$ exist.
\end{proof}

If $m=1$, then we have the line graph of a generalized quadrangle of order $(q-1, q+1)$
with $q=2^h$.
It is not too hard to see that a construction by Wallis \cite{Wallis1971}
produces graphs cospectral
with the line graph of $\cG$ which are not line graphs themselves, 
so Corollary \ref{cor:srg_appl} is surely known for $m=1$.
This note is motivated by \cite{vDG2022} where the authors ask
if there exists an infinite family of non-geometric 
strongly regular graphs cospectral with the line graph of 
a generalized quadrangle of order $(q-1, q+1)$.
More generally, Wallis' construction
with (in Wallis notation) an affine resolvable design of type $\text{AR}(2^h, 1)$
and a block design with $(v,k)=(2^{h+m}+2^m-2^h, 2^m)$ works.
Again, Denniston arcs imply the existence of these structures
(for instance, see \cite{GTW2019}).

When Proposition \ref{prop:switching} is applicable,
then one can most likely apply WQH-switching repeatedly and obtain 
large numbers of graphs. For instance, starting with the line
graph of the unique generalized quadrangle of order $(3,5)$,
so $(h,m) = (2,1)$ in Corollary \ref{cor:srg_appl},
one obtains 133,005 strongly regular graphs
by applying WQH-switching up to six times \cite{Ihringer2020}.

\section{Proof of Proposition \ref{prop:switching}}

Let $M_1, M_2$ be distinct planes of $\AG(n+1, q)$ through $K$.
Pick $P \in \cK \cap K$.
For $i \in \{ 1,2\}$, let $C_i$ denote the lines in 
$M_i \cap \cL$ which contain $P$.

Let us verify that we can apply Lemma \ref{lem:WQH}:
Clearly, $|C_1| = |C_2|$.
Let $L$ be a line of $C_i$ for $\{ i, j \} = \{ 1, 2\}$.
The line $L$ is adjacent to all lines in $C_i$
and no line in $C_j$.
Hence, the induced subgraphs on $C_1$, $C_2$, 
and $C_1 \cup C_2$ are all regular, and the 
induced subgraphs on $C_1$ and $C_2$ have the same orders and degrees.

Now let $L$ be a line of $\cL \setminus (C_1 \cup C_2)$.
If $L \subseteq M_i$ for some $i \in \{1 ,2 \}$, then 
$L$ meets all lines of $C_i$ and none of $C_j$ for $\{ i, j \} = \{ 1, 2\}$.
In all other cases $L$ meets $M_1$  in a point $R_1$ and $M_2$ in a point $R_2$.
Hence, $L$ meets one line of $C_1$ and $C_2$ each.
Hence, we can apply Lemma \ref{lem:WQH}
and obtain a graph $\Gamma'$ cospectral with $\Gamma$.

The discussion above shows that the neighborhood 
of $L \in \cL \setminus ( C_1 \cup C_2 )$
only differs between $\Gamma$ and $\Gamma'$
when $L \subseteq M_i$ and $P \notin L$.
Such lines exist as we require $|L \cap \cK| \geq 2$.

It remains to show that the resulting graph $\Gamma'$ cannot be
the line graph of an incomplete $(s', t, \alpha')$-geometry.
Observe that cliques of $\Gamma$
either have size $t+1$
(when they consist of all lines through point) or size at most $q(\alpha-1)+1$
(when they are contained in a plane of $\AG(n+1, q)$).
Suppose that $\Gamma'$ is the line graph 
of an incomplete $(s', t, \alpha')$-geometry $\cG'$.
Hence, if two lines are adjacent in 
$\Gamma'$, then they lie together in a clique of size $t+1$.

Pick a point $R$ in $M_1$.
For $i \in \{ 1, 2\}$, let $L_i$ be the line through $Q_i$ and $R$.
Note that $L_1$ and $L_2$ are not in $M_1 \cup M_2$, 
so their neighborhoods are the same in $\Gamma$ and $\Gamma'$.
Hence, $L_1, L_2$ are adjacent in $\Gamma$ and $\Gamma'$,
so they lie in a clique of size $t+1$ in $\Gamma$ and $\Gamma'$ each.
For $\Gamma$, this clique is unique (as $t+1 > q(\alpha-1)+1$) and 
consists of all lines through $R$.

The lines $L_1$ and $L_2$ have no common neighbor in $M_2$:
If $L_1$ or $L_2$ does not meet $M_2$, then this is clear.
Hence, suppose that $L_1 \cap M_2$ and $L_2 \cap M_2$ are 
distinct points $S_1$ and $S_2$.
Let $\tilde{L}$ be the line through $S_1$ and $S_2$.
Then $\tilde{L} \cap K = \< S_1, S_2\> \cap K = \< Q_1, Q_2 \> \cap K \notin \cK$.
Hence, $\tilde{L} \notin \cL$.

Now we show that in $\Gamma'$ a clique containing 
$L_1$ and $L_2$ has at most size $t$.

If $L_1, L_2$ lie in a clique $\cC$ of $\Gamma'$
which does not contain a line through $R$, then $|\cC| \leq q(\alpha-1)+1 < t+1$.
Hence, $L_1, L_2$ lie in a clique of size $t+1$ which also contains 
a line $L \in \cL \cap M_1$ with $R \in L$.

If $L \in C_1$,
then let $L'$ be a line of $\cL \setminus C_1$ in $M_1$ with $R \in L'$ 
(which exists as $|L \cap \cK| \geq 2$).
Then $L'$ is nonadjacent to $L$ in $\Gamma'$.
The line $L$ only gains lines in $M_2$ as new neighbors in $\Gamma'$,
but $L_1$ and $L_2$ have no common neighbor in $M_2$ in $\Gamma'$.
Hence, $\{ L, L_1, L_2 \}$ lie in a clique of size at most $t$.

If $L \notin C_1$, then repeat the argument with 
switched roles for $L$ and $L'$, that is $L' \in \cL$
with $R \in L'$.
Hence, $L_1$ and $L_2$ do not lie in a clique of size $t+1$,
so $\Gamma'$ is not the line graph of an incomplete $(s', t, \alpha')$-geometry.

\bigskip
\paragraph*{Acknowledgment}
The first author is supported by a 
postdoctoral fellowship of the Research Foundation -- Flanders (FWO).

\end{document}